\documentclass[12pt]{article}

\usepackage{amssymb}
\usepackage{amsmath}
\usepackage{amsbsy}
\usepackage{amscd}
\usepackage{amsfonts}
\usepackage{amsthm}
\usepackage{mathrsfs}
\usepackage{verbatim}
\usepackage[colorlinks]{hyperref}
\usepackage{fullpage}
\usepackage{mathdots}
\usepackage{graphicx,subfigure}
\usepackage[english]{babel}
\usepackage[utf8]{inputenc}
\usepackage{tikz}

\usepackage{authblk}

\usepackage{mathtext}

\usetikzlibrary{backgrounds,fit, matrix}
\usetikzlibrary{positioning}
\usetikzlibrary{calc,through,chains}
\usetikzlibrary{arrows,shapes,snakes,automata, petri}

\newtheorem{Theorem}[equation]{Theorem}
\newtheorem{Corollary}[equation]{Corollary}
\newtheorem{Lemma}[equation]{Lemma}

\theoremstyle{definition}

\newtheorem{Example}[equation]{Example}

\numberwithin{equation}{section}
\numberwithin{figure}{section}

\newcommand{\C}{{\mathbb C}}

\newcommand{\Q}{{\mathbb Q}}

\newcommand{\mc}[1]{\mathcal{#1}}

\newcommand{\mt}[1]{\text{#1}}

\begin{document}

\title{Adjoint Representations of the Symmetric Group}

\author[1]{Mahir Bilen Can}
\author[2]{Miles Jones}

\affil[1]{{mahirbilencan@gmail.com}}
\affil[2]{{mej016@ucsd.edu}}

\date{}
\maketitle

\begin{abstract}
We study the restriction to the symmetric group, $\mc{S}_n$ 
of the adjoint representation of $\mt{GL}_n(\C)$. 
We determine the irreducible constituents 
of the space of symmetric as well as the space of 
skew-symmetric $n\times n$
matrices as $\mc{S}_n$-modules.

\vspace{.35cm}
\end{abstract}

\vspace{.1cm}
\begin{center}
{\em To the memory of our beloved mentor and a good friend, Jeff Remmel.}
\end{center}
\vspace{.1cm}

\normalsize

\section{Introduction}

In~\cite{CanRemmel}, the first author and Jeff Remmel 
introduced the notion of ``loop-augmented rooted 
forest'' and explained its combinatorial representation 
theoretic role for the conjugation action of the 
symmetric group on certain subsets of the 
partial transformation semigroup. 
In this note, we present an application of this 
development in a basic Lie theory context.
\vspace{.5cm}

Let $G$ be a Lie group and let 
$\mathfrak{g}$ denote the Lie 
algebra of $G$. 
The conjugation action $G\times G\rightarrow G$,
$(g,h) \mapsto ghg^{-1}$, $g,h\in G$
leads to a linear representation of $G$
on its tangent space at the identity element,
\begin{align}
\text{Ad} : G & \rightarrow \text{Aut}(\mathfrak{g}) \notag \\
g &\mapsto \text{Ad}_g.\label{A:adjoint}
\end{align}
The representation (\ref{A:adjoint}), 
which is called the adjoint representation, 
has a fundamental place in the 
structure theory of Lie groups. 
It has a concrete description when $G$ 
is a closed subgroup of $\mt{GL}_n(\C)$, 
the general linear group of $n\times n$ matrices. 
In this case, the Lie algebra $\mathfrak{g}$ 
of $G$ is a Lie subalgebra of the $n\times n$
matrices, and the adjoint representation 
(\ref{A:adjoint}) is given by 
$$
\text{Ad}_g ( X) = g X g^{-1},
$$ 
for $g\in G,\ X\in \mathfrak{g}$.
\vspace{.5cm}

Let $\mc{S}_n$ denote the group 
of permutations of the set $\{1,\dots, n\}$. 
We view $\mc{S}_n$ as a subgroup 
of $\mt{GL}_n(\C)$ by identifying its
elements with $n\times n$ 0/1 matrices
with at most one 1 in each row and each column.
The basic representation theoretic 
question that we address here is the
following: 
\begin{quote}
What are the irreducible constituents of the 
$S_n$-representation that is obtained from 
(\ref{A:adjoint}) by restriction? 
\end{quote}

Surprisingly, even though the adjoint representation
is at the heart of Lie theory, to the best of our knowledge
the answer to our question is missing from the literature, 
at least, it is not presented in the way that we are answering it. 
To state our theorem we set up the notation.

It is well known that the finite dimensional 
irreducible representations over $\Q$ of $\mc{S}_n$
are indexed by the integer partitions of $n$. 
The Frobenius character map, $V\mapsto F_V$ 
is an assignment of symmetric functions to the 
finite dimensional representations of $\mc{S}_n$. 
(We will explain this in more detail in the sequel.)
In particular, 
if $V$ is the irreducible representation determined 
by an integer partition $\lambda$, then 
$F_V$ is a Schur symmetric function,
denoted by $s_\lambda$. 
Furthermore, if $V= \bigoplus V_i$ is a decomposition
of $V$ into $\mc{S}_n$-submodules, then 
$F_V = \sum F_{V_i}$. 
Our first main result is as follows.

\begin{Theorem}\label{T:main1}
Let $n$ be an integer such that $n\geq 2$. 
The Frobenius character of the adjoint representation of $\mc{S}_n$ on 
$\mt{Mat}_n(\C)$ is given by 
\begin{align}
F_{\mt{Mat}_n(\C)} =
\begin{cases}
2s_2 + 2s_{1,1} & \text{ if $n=2$}; \\
2s_3+3s_{2,1}+s_{1,1,1} & \text{ if $n=3$};\\
2s_n + 3s_{n-1,1} + s_{n-2,2} + s_{n-2,1,1} & \text{ if $n\geq 4$}.
\end{cases}
\end{align}
\end{Theorem}

\vspace{.5cm}

The space of symmetric $n\times n$-matrices, 
which we denote by $\mt{Sym}_n(\C)$, is closed
under the adjoint action of the orthogonal group, 
$\mt{O}_n(\C):=\{ g \in \mt{GL}_n(\C):\ g g^\top = id \}$.
However, $\mt{Sym}_n(\C)$ is not closed under 
the adjoint action of $\mt{GL}_n(\C)$.
Nevertheless, since $\mc{S}_n$ is a subgroup of 
$\mt{O}_n(\C)$, we see that the representation 
\begin{align}\label{A:adjoint on sym}
\mt{Ad}: \mc{S}_n \rightarrow \mt{Aut}( \mt{Sym}_n(\C))
\end{align}
is defined. Moreover, since there is a direct 
sum decomposition
$$
\mt{Mat}_n(\C) = \mt{Sym}_n(\C) \oplus \mt{Skew}_n(\C),
$$
where $\mt{Skew}_n(\C)$ is the space of $n\times n$
skew-symmetric matrices, we have the complementary 
adjoint representation
\begin{align}\label{A:adjoint on skew}
\mt{Ad}: \mc{S}_n \rightarrow \mt{Aut}( \mt{Skew}_n(\C))
\end{align}
as well.

Our second main result is the following

\begin{Theorem}\label{T:main2}
Let $n$ be an integer such that $n\geq 2$. 
The Frobenius character of the adjoint representation (\ref{A:adjoint on sym}) of $\mc{S}_n$ on 
$\mt{Sym}_n(\C)$ is given by 
\begin{align}
F_{\mt{Sym}_n(\C)} =
\begin{cases}
2s_2 + s_{1,1} & \text{ if $n=2$}; \\
2s_3+3s_{2,1} & \text{ if $n=3$};\\
2s_n + 2s_{n-1,1} + s_{n-2,2} & \text{ if $n\geq 4$}.
\end{cases}
\end{align}
\end{Theorem}

\begin{Corollary}\label{T:main3}
Let $n$ be an integer such that $n\geq 2$. 
The Frobenius character of the adjoint representation 
(\ref{A:adjoint on skew}) of $\mc{S}_n$ on 
$\mt{Skew}_n(\C)$ is given by 
\begin{align}
F_{\mt{Skew}_n(\C)} =
\begin{cases}
s_{1,1} & \text{ if $n=2$}; \\
s_{1,1,1} & \text{ if $n=3$};\\
s_{n-1,1} + s_{n-2,1,1} & \text{ if $n\geq 4$}.
\end{cases}
\end{align}
\end{Corollary}

\section{Preliminaries}


\subsection{Symmetric functions and plethysm.}

The {\em   $ k $-th power-sum symmetric function}, denoted by $ p_k $, is the sum of $ k $-th powers of the 
variables. For an integer partition $\lambda = (\lambda_1 \geq \ldots \geq \lambda_l)$, we define $p_\lambda$ to be 
equal to $\prod_i p_{\lambda_i} $. 
The {\em $k$-th complete symmetric function}, $h_k$, is defined to be the 
sum of all monomials $x_{i_1}^{a_1}\cdots x_{i_r}^{a_r}$ with $ \sum a_i =k $ and $ h_\lambda$ is defined to 
be equal to $\prod_{i=1}^l h_{\lambda_i}$. The {\em Schur function associated with $ \lambda$}, denoted by 
$s_\lambda$, is the symmetric function defined by the determinant $\det (h_{\lambda_i+j - i})_{ i,j=1}^l$. 
In particular, we have, for all $k\geq 1$, that $s_{(k)} = h_k$. For easing the notation, we denote $s_{(k)}$ by
$s_k$. 

The type of a conjugacy class $\sigma$ in $\mc{S}_n$ is the partition $\lambda$ 
whose parts correspond to the lengths of the cycles that 
appear in the cycle decomposition of an element $x\in \sigma$. 
It is well known that the type is independent of the element $x$, furthermore, 
it uniquely determines the conjugacy class.

The 
power sum and Schur symmetric functions will play special roles in our computations via the {\em Frobenius 
character} map 
\begin{align*}
F : \text{ class functions on $\mc{S}_n$ } & \rightarrow  \text{ symmetric functions} \\
\delta_\sigma & \mapsto \frac{1}{n!}\ p_\lambda,
\end{align*}
where $ \sigma \subset \mc{S}_n $ is a conjugacy class of type $ \lambda $ and $ \delta_\sigma$ is the indicator 
function
\begin{align*}
\delta_\sigma( x) =
\begin{cases}
1 & \text{ if } x\in \sigma;\\
0 & \text{ otherwise.}
\end{cases}
\end{align*}
It turns out that if $ \chi^\lambda$ is the irreducible character of $ \mc{S}_n $ indexed by the partition $ \lambda $, 
then $F (\chi^\lambda) = s_\lambda $. In the sequel, we will not distinguish between representations 
of $ \mc{S}_n $ and their corresponding characters. In particular, we will often write the {\em Frobenius character of an 
orbit} to mean the image under $F $ of the character of the representation of $ \mc{S}_n $ that is defined 
by the action on the orbit. If $V$ is an $\mc{S}_n$-representation, then we will denote its Frobenius character by $F_V$.

For two irreducible characters $ \chi^\mu$ and $ \chi^\lambda$ indexed by partitions $ \lambda $ and $\mu$,  
the Frobenius character image of the ``plethysm'' $\chi^\lambda [ \chi^\mu]$ is the plethystic substitution 
$s_\lambda [ s_\mu]$ of the corresponding Schur functions. Roughly speaking, the plethysm of the Schur 
function $s_\lambda$ with $s_\mu$ is the symmetric function obtained from $s_\lambda$ by substituting the 
monomials of $s_\mu$ for the variables of $s_\lambda$. In the notation of~\cite{LoehrRemmel}; the {\em plethysm} 
of symmetric functions is the unique map $[\cdot ]:\ \Lambda\times \Lambda \rightarrow \Lambda$ satisfying the 
following three axioms: 
\begin{enumerate}
\item[P1.] For all $m,n\geq 1$, $p_m [p_n] = p_{mn}$.
\item[P2.] For all $m\geq 1$, the map $g \mapsto p_m [ g]$, $g\in \Lambda$ defines a $\Q$-algebra 
homomorphism on $\Lambda$. 
\item[P3.] For all $g\in \Lambda$, the map $h \mapsto h [ g]$, $h\in \Lambda$ defines a $\Q$-algebra 
homomorphism on $\Lambda$. 
\end{enumerate}

Although the problem of computing the plethysm of two (arbitrary) symmetric functions is very difficult, there 
are some useful formulas for Schur functions:
\begin{align}\label{A:plethysm formula 1}
s_\lambda [ g+h] &= \sum_{\mu,\nu} c_{\mu,\nu}^\lambda (s_\mu [ g]) (s_\nu [ h]),
\end{align}
and 
\begin{align}\label{A:plethysm formula 2}
s_\lambda [gh] &= \sum_{\mu,\nu} \gamma_{\mu,\nu}^\lambda (s_\mu [ g]) (s_\nu [h]).
\end{align}
Here, $ g$ and $ h$ are arbitrary symmetric functions, $ c_{\mu,\nu}^\lambda $ is a scalar, and $\gamma_{
\mu,\nu}^\lambda$ is $\frac{1}{n!} \langle \chi^\lambda, \chi^\mu \chi^\nu \rangle$, where the pairing stands
for the standard Hall inner product on characters.

In~(\ref{A:plethysm formula 1}) the summation is over all pairs of partitions $\mu,\nu \subset \lambda$, and 
the summation in~(\ref{A:plethysm formula 2}) is over all pairs of partitions $\mu,\nu$ such that $  |\mu| =  |
\nu | = |\lambda|$. In the special case when $\lambda = (n)$, or $(1^n)$ we have
\begin{align}
s_{(n)} [gh] &= \sum_{\lambda \vdash n} (s_\lambda [g]) (s_\lambda [h]), \label{A:complete}\\
s_{(1^n)} [gh]&= \sum_{\lambda \vdash n}(s_\lambda[g])(s_{\lambda'}[h]), \label{A:elementary}
\end{align}
where $\lambda'$ denotes the conjugate of $\lambda$.

\subsection{Background on partial transformations.}

Let $n$ denote a positive integer, $[n]$ and $\overline{[n]}$ denote the sets $\{1,\dots, n\}$ and $[n]\cup 
\{0\}$, respectively. We will use the following basic notation for our semigroups: 
\begin{center}
\begin{tabular}{c c l}
$\mc{F}ull_n$ &:& the full transformation semigroup on $\overline{[n]}$; \\
$\mc{P}_n$ &:& the semigroup of partial transformations on $[n]$; \\
$\mc{C}_n$ &:& the set of nilpotent partial transformations on $[n]$.
\end{tabular}
\end{center}
A partial transformation on $[n]$ is a function $f: A \rightarrow [n] $, where $A$ is a nonempty 
subset of $[n]$. A {\em full transformation} on $\overline{[n]}$ is a function $g:\overline{[n]} \rightarrow 
\overline{[n]}$. We note that there is an ``extension by 0'' morphism from partial transformations on $[n]$ into full 
transformations on $\overline{[n]}$,
\begin{align*}
\varphi_0 \ : \ & \mc{P}_n  \rightarrow \mc{F}ull_n\\ 
& f  \longmapsto  \varphi_0(f)
\end{align*} 
which is defined by
$$
\varphi_0(f)(i) = 
\begin{cases}
f(i) & \text{ if $i$ is in the domain of $f$}; \\
0 & \text{ otherwise.}
\end{cases}
$$
The map $ \varphi$ is an injective semigroup homomorphism. 
Since $ \mc{F}ull_n $ contains a zero transformation it makes sense to talk about 
nilpotent partial transformations in $\mc{P}_n$. In particular, $\mc{C}_n$ 
is well defined although its definition requires the embedding of $\mc{P}_n$ into $\mc{F}ull_n$. 
The unit group of $\mc{P}_n$ is equal to $\mc{S}_n$, therefore, its 
conjugation action of $\mc{P}_n$ makes sense. Furthermore, the set of 
nilpotent partial transformations is stable under this action of $\mc{S}_n$. 

The combinatorial significance of $\mc{C}_n$ stems from the fact that there 
is a bijection between $\mc{C}_n$ and the set of labeled rooted forests on $ n$ vertices. 
(Hence, $|\mc{C}_n|=(n+1)^{n-1}$.)
The conjugation action of $\mc{S}_n$ on $\mc{C}_n$ 
translates into the permutation action of $\mc{S}_n$ on the labels. 
This idea, which we started to use in~\cite{Can17}, extends in a rather natural way
to all directed graphs that correspond to the elements of $\mc{P}_n$. 
Indeed, let $\tau$ be a partial transformation from $\mc{P}_n$. 
We view $\tau$ as an $n\times n$ 0/1 matrix with at most one 1 in each column,
so, it defines a directed, labeled graph on $n$ vertices; 
there is a directed edge from the $i$-th vertex
to the $j$-th vertex if the $(i,j)$-th entry of the matrix is 1. The underlying graph of this labeled 
directed graph depends only on the $\mc{S}_n$-conjugacy class of $\tau$. 
To explain the consequences of this identification, 
next, we will re-focus on the elements of $\mc{C}_n$ 
and on the corresponding labeled rooted forests. What we are going to state for $\mc{C}_n$ 
extends to all partial transformations and to the corresponding labeled directed graphs.

A pair $(\tau,\phi)$, where $\tau$ is a rooted forest on $n$ vertices and $ \phi $ is a bijective map  
from $[n]$ onto the vertex set of $ \tau $ is called a labeled rooted forest. 
As a convention, when we talk about labeled rooted forests, 
we will omit writing the corresponding labeling function despite the fact that the action of $\mc{S}_n $ 
does not change the underlying forest but the labeling function only. In particular, 
when we write $ \mc{S}_n\cdot \tau$ we actually mean the orbit 
\begin{align}\label{A:an orbit}
\mc{S}_n \cdot (\tau,\phi) = \{ (\tau,\phi'):\ \phi' = \sigma \cdot \phi,\ \sigma \in \mc{S}_n \}.
\end{align}
The right hand side of (\ref{A:an orbit}) is an $\mc{S}_n$-set, hence it defines a representation of $\mc{S}_n$. 
More generally, to any partial transformation $\tau$ in $\mc{P}_n$, we associate the representation 
corresponding to the orbit $\mc{S}_n \cdot \tau$. We refer to the resulting representation by the odun 
of $\tau$, and denote it by $o(\tau)$. 

The odun depends only on $\tau$, not on the labels, so, we write $ \mt{Stab}_{\mc{S}_n} (\tau)$ 
to denote the stabilizer 
subgroup of the pair $(\tau,\phi) $. As an $\mc{S}_n$-module, the vector space of functions on the right 
cosets, that is $\C[ \mc{S}_n/\text{Stab}_{ \mc{S}_n }(\tau)] $ is isomorphic to the odun of $ \tau$. 
As a representation of $\mc{S}_n$, this is  
equivalent to the induced representation $ \mt{Ind}_{\mt{Stab}_{\mc{S}_n} (\tau)}^{\mc{S}_n}\mathbf{1}$. 
Next, following~\cite{Can17}, we present an example to demonstrate the computation of the Frobenius character 
of $ \mt{Ind}_{\mt{Stab}_{\mc{S}_n} (\tau)}^{\mc{S}_n}\mathbf{1}$.
\begin{Example}
Let $ \tau$ be the rooted forest depicted in Figure~\ref{F:Example1} and let $ \tau_i $, $ i=1, 2, 3 $ denote 
its connected components (from left to right in the figure). 
\begin{figure}[htp]
\centering
\begin{tikzpicture}[scale=.75]
\begin{scope}[xshift=3cm]
\node at (0,0) {$\bullet$};
\node at (1,1) {$\bullet$};
\node at (-1,1) {$\bullet$};
\node at (1,2) {$\bullet$};
\node at (2,3) {$\bullet$};
\node at (0,3) {$\bullet$};
\node at (-1,2) {$\bullet$};
\draw[-, thick] (0,0) to  (1,1);
\draw[-, thick] (0,0) to  (-1,1);
\draw[-, thick] (-1,1) to  (-1,2);
\draw[-, thick] (1,1) to  (1,2);
\draw[-, thick] (1,2) to  (2,3);
\draw[-, thick] (1,2) to  (0,3);
\end{scope}
\begin{scope}[xshift=0cm]
\node at (0,0) {$\bullet$};
\node at (1,1) {$\bullet$};
\node at (-1,1) {$\bullet$};
\node at (1,2) {$\bullet$};
\node at (2,3) {$\bullet$};
\node at (0,3) {$\bullet$};
\node at (-1,2) {$\bullet$};
\draw[-, thick] (0,0) to  (1,1);
\draw[-, thick] (0,0) to  (-1,1);
\draw[-, thick] (-1,1) to  (-1,2);
\draw[-, thick] (1,1) to  (1,2);
\draw[-, thick] (1,2) to  (2,3);
\draw[-, thick] (1,2) to  (0,3);
\end{scope}
\begin{scope}[xshift=-3.5cm]
\node at (0,0) {$\bullet$};
\node at (0,1) {$\bullet$};
\node at (-.5,2) {$\bullet$};
\node at (-1.5,2) {$\bullet$};
\node at (.5,2) {$\bullet$};
\node at (1.5,2) {$\bullet$};
\draw[-, thick] (0,0) to  (0,1);
\draw[-, thick] (0,1) to  (-1.5,2);
\draw[-, thick] (0,1) to  (1.5,2);
\draw[-, thick] (0,1) to  (-.5,2);
\draw[-, thick] (0,1) to  (.5,2);
\end{scope}
\end{tikzpicture}
\caption{An example.}
\label{F:Example1}
\end{figure}
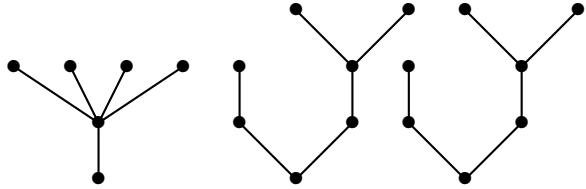
Let $F_{o(\tau)}$ and $F_{o(\tau_i)}$, $i=1,2,3$ denote the corresponding Frobenius characters. Since
$\tau_1 \neq \tau_2 = \tau_3$, we have 
\begin{align}\label{A:F0}
F_{o(\tau)} &= F_{o(\tau_1)} \cdot s_2[ F_{o(\tau_2)}].
\end{align}
Here, $ s_2 $ is the Schur function $ s_{(2)} $, the bracket stands for the plethysm of symmetric functions and 
the dot stands for ordinary multiplication. More generally, if 
a connected component $\tau'$, which of course is a rooted tree appears $k$-times in a forest $ \tau $, then 
$F_{o(\tau)}$ has $s_k [ F_{o(\tau')}]$ as a factor. Now we proceed to explain the computation of the 
Frobenius character of a rooted tree. As an example we use $ F_{o(\tau_1)} $ of Figure~\ref{F:Example1}. 
The combinatorial rule that we obtained in~\cite{Can17} is simple; it is the removal of the root from the tree. The 
effect on the Frobenius character of this simple rule is as follows: Let $ \tau_1'$ denote the rooted forest that 
we obtain from $ \tau_1$ by removing the root. Then $ F_{o(\tau_1)}  = s_1 \cdot F_{o(\tau_1')} $. Thus, 
by the repeated application of this rule and the previous factorization rule, we obtain 
$F_{o(\tau_1)} = s_1\cdot F_{o(\tau_1')}= s_1\cdot s_1 \cdot s_4[ s_1]$.
It follows from the definition of plethysm that $s_k[s_1]=s_k$ for any nonnegative integer $k$.
Therefore, 
\begin{align}\label{A:F1}
F_{o(\tau_1)} = s_1^2\cdot s_4.
\end{align}
We compute $ F_{o(\tau_2)} $ by the same method; 
\begin{align}\label{A:F2}
F_{o(\tau_2)}= s_1^5 \cdot s_2.
\end{align}
By putting (\ref{A:F0})--(\ref{A:F2}) together, we arrive at the following satisfactory expression 
for the Frobenius character of $\tau$:
\begin{align}\label{A:F character of tau}
F_{o(\tau)}&= F_{o(\tau_1)} \cdot s_2[F_{o(\tau_2)}]= s_1^2\cdot s_4\cdot s_2[s_1^5 \cdot s_2].
\end{align}
Note that the expansion of $s_2[s_1^5 \cdot s_2]$ in the Schur basis is computable
by a recursive method by applying~(\ref{A:complete}) and using Thrall's formula~\cite{Thrall} 
(see~\cite[Chapter I, Section 8, Example 9]{Mac}).
However, the resulting expression is rather large, so, we omit writing it here.
\end{Example}

\vspace{.5cm}
A {\em loop-augmented forest} is a rooted forest such that there is at most one loop at each of its roots. See, for 
example, Figure~\ref{F:pruned}, where we depict a loop-augmented forest on 22 vertices and four loops.
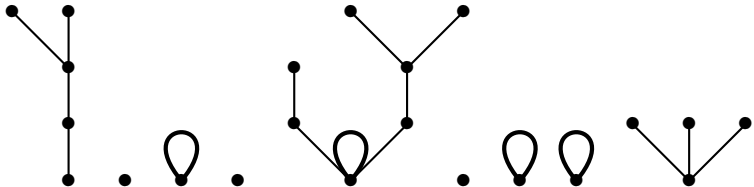
\begin{figure}[htp]
\centering
\begin{tikzpicture}[scale=.75]
\begin{scope}
\node at (0,0) {$\bullet$};
\node at (1,1) {$\bullet$};
\node at (-1,1) {$\bullet$};
\node at (1,2) {$\bullet$};
\node at (2,3) {$\bullet$};
\node at (0,3) {$\bullet$};
\node at (-1,2) {$\bullet$};
\draw[-, thick] (0,0) to  (1,1);
\draw[-, thick] (0,0) to  (-1,1);
\draw[-, thick] (-1,1) to  (-1,2);
\draw[-, thick] (1,1) to  (1,2);
\draw[-, thick] (1,2) to  (2,3);
\draw[-, thick] (1,2) to  (0,3);
\end{scope}
\begin{scope}[scale=3]
\draw[ultra thick] (0,0)  to[in=50,out=130, loop] (0,0);
\end{scope}
\begin{scope}[xshift=-5cm]
\node at (0,0) {$\bullet$};
\node at (0,1) {$\bullet$};
\node at (0,2) {$\bullet$};
\node at (0,3) {$\bullet$};
\node at (-1,3) {$\bullet$};
\draw[-, thick] (0,0) to  (0,1);
\draw[-, thick] (0,1) to  (0,2);
\draw[-, thick] (0,2) to  (0,3);
\draw[-, thick] (0,2) to  (-1,3);
\end{scope}
\begin{scope}[xshift=-4cm]
\node at (0,0) {$\bullet$};
\end{scope}
\begin{scope}[xshift=-3cm,scale=3]
\node at (0,0) {$\bullet$};
\draw[ultra thick] (0,0)  to[in=50,out=130, loop] (0,0);
\end{scope}
\begin{scope}[xshift=-2cm]
\node at (0,0) {$\bullet$};
\end{scope}
\begin{scope}[xshift=2cm]
\node at (0,0) {$\bullet$};
\end{scope}
\begin{scope}[xshift=3cm,scale=3]
\node at (0,0) {$\bullet$};
\draw[ultra thick] (0,0)  to[in=50,out=130, loop] (0,0);
\end{scope}
\begin{scope}[xshift=4cm,scale=3]
\node at (0,0) {$\bullet$};
\draw[ultra thick] (0,0)  to[in=50,out=130, loop] (0,0);
\end{scope}
\begin{scope}[xshift=6cm]
\node at (0,0) {$\bullet$};
\node at (1,1) {$\bullet$};
\node at (-1,1) {$\bullet$};
\node at (0,1) {$\bullet$};
\draw[-, thick] (0,0) to  (1,1);
\draw[-, thick] (0,0) to  (-1,1);
\draw[-, thick] (0,0) to  (0,1);
\end{scope}
\end{tikzpicture}
\caption{A loop-augmented forest.}
\label{F:pruned}
\end{figure}
It is a well known variation of the Cayley's theorem that the number of labeled forests on $ n$ vertices with 
$k$ roots is equal to ${n-1\choose k-1} n^{n-k}$. See~\cite{MR0460128}, Theorem D, pg 70. It follows that
the number of loop-augmented forests on $ n $ vertices with $ k $ roots is $2^k{n-1\choose k-1} n^{n-k} $. 
It follows from generating function manipulations that the number of loop-augmented forests on $n$ vertices 
for $ n\geq 2$ is $2n^{n-3}$.

Next, we will explain how to interpret loop-augmented forests in terms of partial functions. Let $\sigma$ be 
a loop-augmented forest. Then some of the roots of $\sigma$ have loops. There is still a partial function for 
$\sigma$, as defined in the previous paragraph for a rooted forest. The loops in this case correspond to the 
``fixed points'' of the associated function. Indeed, for the loop at the $i$-th vertex we have a 1 at the $(i,i)$-th
entry of the corresponding matrix representation of the partial transformation. 
Recall that the permutation action of $ \mc{S}_n $ on the labels translates to the conjugation action on the 
incidence matrix. By an appropriate relabeling of the vertices, the incidence matrix of a loop-augmented 
rooted forest can be brought to an upper-triangular form.
Clearly, the conjugates of a nilpotent (respectively unipotent) matrix are still nilpotent 
(respectively unipotent). Let $U$ be an arbitrary upper triangular matrix. Then $U$ is equal to a sum of the 
form $ D+N $, where $ D$ is a diagonal matrix and $N$ is a nilpotent matrix. If $ \sigma $ is a permutation matrix 
of the same size as $ U$, then we conclude from the equalities $\sigma \cdot U= \sigma U\sigma^{-1}= \sigma 
D \sigma^{-1}  + \sigma N \sigma^{-1}$ that the conjugation action on loop-augmented forests is equivalent to 
the simultaneous conjugation action on labeled rooted forests and the representation of $\mc{S}_n$ on 
diagonal matrices.

The following two theorems from~\cite{CanRemmel} 
describe, respectively, 
the stabilizer of a labeled loop-augmented rooted 
forest and the character of the corresponding odun.

\begin{Theorem}\label{T:Stabilizer}
Let $f$ be a partial transformation of the form $f = \begin{pmatrix} \sigma & 0 \\ 0 & \tau \end{pmatrix}$, where 
$\sigma \in \mc{S}_k$ is a permutation and $\tau\in \mc{C}_{n-k}$ is a nilpotent partial transformation. In this
case, the stabilizer subgroup in $\mc{S}_n$ of $f$ has the following decomposition:
$$
\mt{Stab}_{\mc{S}_n} (f) = Z(\sigma) \times \mt{Stab}_{\mc{S}_{n-k}}(\tau),
$$
where $Z(\sigma)$ is the centralizer of $\sigma$ in $\mc{S}_k$ and $\mt{Stab}_{\mc{S}_{n-k}}(\tau)$ is the stabilizer 
subgroup of $\tau$ in $\mc{S}_{n-k}$.
\end{Theorem}

\begin{Theorem}\label{T:Master 1}
Let $f$ be a partial function representing a loop-augmented forest on $n$ vertices. 
Then $f$ is similar to a block diagonal matrix of the form $\begin{pmatrix}  \sigma & 0 \\ 0 & \tau 
\end{pmatrix} $ for some $ \sigma \in \mc{S}_k $ and $ \tau \in \mc{C}_{n-k}$. Furthermore, 
if $ \nu$, which is a partition of $k$, is the conjugacy type of $\sigma$ in $\mc{S}_k$ and if the underlying rooted forest of 
the nilpotent partial transformation $\tau$ has $\lambda_1$ copies of the rooted tree $\tau_1$, 
$\lambda_2$ copies of the rooted forest $\tau_2$ and so on, then the character of $ o(f)$ is given by 
\begin{align*}
\chi_{o(f)}  =  \chi^\nu \cdot  \chi_{o(\tau)} = \chi^\nu  \cdot  (\chi^{(\lambda_1)} [ \chi_{o(\tau_1)}]) \cdot (\chi
^{(\lambda_2)} [ \chi_{o(\tau_2)}])\cdot \cdots \cdot (\chi^{(\lambda_r)} [\chi_{o(\tau_r)}]).
\end{align*}
\end{Theorem}

\section{Proof of Theorem~\ref{T:main1}}

Let $i$ and $j$ be two integers from $[n]$. 
We denote by $E_{i,j}$ the $n\times n$ 0/1 matrix with
1 at its $(i,j)$-th entry and 0's elsewhere. 
As a vector space, $\mt{Mat}_n(\C)$ is spanned by $E_{i,j}$'s.
In fact, $\{E_{i,j} :\ i,j\in [n] \}$ constitute a basis,
\begin{align}\label{A:direct sum}
\mt{Mat}_n(\C) = \bigoplus_{i,j\in [n]} \C E_{i,j}. 
\end{align}
It is clear that $E_{i,j}$'s are actually partial transformation matrices.
Equivalently, we will view $E_{i,j}$'s as labeled loop-augmented rooted forests,
essentially in two types. 
The adjoint representation
of $S_n$ on $\text{Mat}_n(\C)$ is completely determined
by the action of $S_n$ on these two types of 
labeled loop-augmented rooted forests;

\begin{enumerate}
\item $i,j \in [n]$ and $i \neq j$. 
In this case, the labeled loop-augmented rooted forest 
corresponding to $E_{i,j}$ is as in Figure~\ref{F:C1}.
\begin{figure}[h]
\centering
\begin{tikzpicture}
[scale=.75,cap=round,>=latex,distance=2cm, thick]
\node[inner sep =0.5, circle,draw,ultra thick] (1) at (-3,0) {1};
\node[inner sep =0.5, circle,draw,ultra thick] (2) at (-2,0) {2};
\node (3) at (-1,0) {$\dots$};
\node[inner sep =0.65, circle,draw,ultra thick] (4) at (0,0) {$i$};
\node (5) at (1,0) {$\dots$};
\node[inner sep =0.5, circle,draw,ultra thick] (6) at (2,0) {$j$};
\node (7) at (3,0) {$\dots$};
\node[inner sep =0.5, circle,draw,ultra thick] (8) at (4,0) {$n$};
\draw[ultra thick,->] (4) to [out = 45, in =135,distance = 1cm, looseness=.5] (6);
\end{tikzpicture}
\caption{The basis element $E_{i,j}$ as a labeled loop-augmented rooted forest.}
\label{F:C1}
\end{figure}
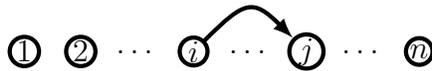

By~\cite[Corollary 6.2]{Can17}, the Frobenius character of the $\mc{S}_n$-module 
structure on the orbit $\mc{S}_n \cdot E_{i,j}$ is given by
\begin{align*}
F_{o (E_{i,j})} = s_1\cdot s_1 \cdot s_{n-2}[s_1] = s_1^2 s_{n-2}.
\end{align*}
By applying the Pierri rule (twice), we find that 
\begin{align}\label{A:type 1}
F_{o (E_{i,j})} = 
\begin{cases}
s_2+s_{1,1} & \text{ if $n=2$}; \\
s_3+2s_{2,1}+s_{1,1,1} & \text{ if $n=3$}; \\
s_n + 2s_{n-1,1} + s_{n-2,2} + s_{n-2,1,1}& \text{ if $n\geq 4$}.
\end{cases}
\end{align}

\item $i,j \in [n]$ and $i= j$. 
In this case, the labeled loop-augmented rooted forest 
corresponding to $E_{i,i}$ is as in Figure~\ref{F:C2}.
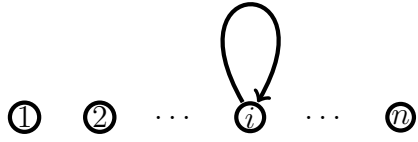
\begin{figure}[h]
\centering
\begin{tikzpicture}
\node[inner sep =0.5, circle,draw,ultra thick] (1) at (-3,0) {1};
\node[inner sep =0.5, circle,draw,ultra thick] (2) at (-2,0) {2};
\node (3) at (-1,0) {$\dots$};
\node[inner sep =0.65, circle,draw, ultra thick] (4) at (0,0) {$i$};
\node (5) at (1,0) {$\dots$};
\node[inner sep =0.5, circle,draw,ultra thick] (8) at (2,0) {$n$};
\draw[ultra thick, ->] (-.1,0.2) to [controls=+(-60:-2) and +(60:2)] (0.1,0.2);
\end{tikzpicture}
\caption{The basis element $E_{i,i}$ as a labeled loop-augmented rooted forest.}
\label{F:C2}
\end{figure}

By Theorem~\ref{T:Master 1}, the Frobenius character of the $\mc{S}_n$-module 
structure on the orbit $\mc{S}_n \cdot E_{i,i}$ is given by
\begin{align*}
F_{o (E_{i,i})} = s_1 \cdot s_{n-1}[s_1] = s_1 s_{n-1}.
\end{align*}
By Pierri rule, we see that 
\begin{align}\label{A:type 2}
F_{o (E_{i,j})} = s_n + s_{n-1,1}. 
\end{align}

\end{enumerate}

Note that $E_{k,l}\in \mc{S}_n\cdot E_{i,j}$ for all
$k,l\in [n]$ with $k\neq l$. Indeed, $\mc{S}_n$ 
acts on $E_{i,j}$ by permuting the labels on the 
vertices. Note also that $E_{k,k}\in \mc{S}_n\cdot E_{i,i}$ 
for all $k\in [n]$. Therefore, in the light of direct sum
(\ref{A:direct sum}), by combining (\ref{A:type 1}) and (\ref{A:type 2}), 
we see that the Frobenius character of the adjoint representation
of $\mc{S}_n$ on $\mt{Mat}_n(\C)$ is as we claimed in Theorem~\ref{T:main1}.
This finishes the proof of our first main result.

\section{Proofs of Theorem~\ref{T:main2} and Corollary~\ref{T:main3}}

First, we have some remarks about the 
vector space basis for $\mt{Sym}_n(\C)$. 
For $i,j\in [n]$ with $i\neq j$, we set 
$$
F_{i,j} := E_{i,j} + E_{j,i}.
$$
A vector space basis for $\mt{Sym}_n(\C)$ is 
given by the union  
\begin{align}\label{A:partial involution}
\{ E_{i,i}:\ i=1,\dots,n \} \cup \{ F_{i,j} :\ i,j\in [n],\ i\neq j \}.
\end{align}
Notice that the matrices $F_{i,j}$'s are partial transformation matrices as well. 
However, this time, the directed graph corresponding to $F_{i,j}$ is not a forest.
See Figure~\ref{F:I1}.
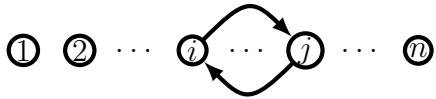
\begin{figure}[h]
\centering
\begin{tikzpicture}
[scale=.75,cap=round,>=latex,distance=2cm, thick]
\node[inner sep =0.5, circle,draw,ultra thick] (1) at (-3,0) {1};
\node[inner sep =0.5, circle,draw,ultra thick] (2) at (-2,0) {2};
\node (3) at (-1,0) {$\dots$};
\node[inner sep =0.65, circle,draw,ultra thick] (4) at (0,0) {$i$};
\node (5) at (1,0) {$\dots$};
\node[inner sep =0.5, circle,draw,ultra thick] (6) at (2,0) {$j$};
\node (7) at (3,0) {$\dots$};
\node[inner sep =0.5, circle,draw,ultra thick] (8) at (4,0) {$n$};
\draw[ultra thick,->] (4) to [out = 45, in =135,distance = 1cm, looseness=.5] (6);
\draw[ultra thick,->] (6) to [out = 225, in =315,distance = 1cm, looseness=.5] (4);
\end{tikzpicture}
\caption{The basis element $F_{i,j}$ as a directed graph.}
\label{F:I1}
\end{figure}

\begin{Lemma}\label{L:perm on Fij}
Let $i$ and $j$ be two elements from $[n]$ such that $i\neq j$. 
The orbit of the adjoint action of $\mc{S}_n$ on the matrix 
$F_{i,j}$ is the same as the permutation action of 
$\mc{S}_n$ on the set of all labelings of the vertices of the 
directed graph in Figure~\ref{F:directed}.
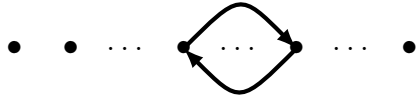
\begin{figure}[h]
\centering
\begin{tikzpicture}
[scale=.75,cap=round,>=latex,distance=2cm, thick]
\draw[ultra thick, ->] (0,0) .. controls (1,1)  .. (2,0);
\draw[ultra thick, ->] (2,0) .. controls (1,-1)  .. (0,0);
\node (1) at (-3,0) {$\bullet$};
\node (2) at (-2,0) {$\bullet$};
\node (3) at (-1,0) {$\dots$};
\node (4) at (0,0) {$\bullet$};
\node (5) at (1,0) {$\dots$};
\node (6) at (2,0) {$\bullet$};
\node (7) at (3,0) {$\dots$};
\node (8) at (4,0) {$\bullet$};
\end{tikzpicture}
\caption{The arrows are between the $i$-th and the $j$-th vertices.}
\label{F:directed}
\end{figure}
In particular, $F_{1,2} \in \mc{S}_n\cdot F_{i,j}$.
\end{Lemma}

\begin{proof}
Since $i$ and $j$ are two different but otherwise arbitrary elements from $[n]$, 
it suffices to prove our second claim only. 
Also, without loss of generality we will assume that $i < j$. 
Now, we apply the adjoint action of the transposition $(1,j)$ to $F_{i,j}$;
$$
\mt{Ad}_{(1,j)} (F_{i,j}) = F_{1,j}.
$$
Next, we apply the adjoint action of the transposition $(2,j)$ to $F_{1,j}$;
$$
\mt{Ad}_{(2,j)} (F_{1,j}) = F_{1,2}.
$$
Therefore, $\mt{Ad}_{(2,j)(1,j)} (F_{i,j})= F_{1,2}$. This finishes the proof.
\end{proof}

In the notation of~\cite{CanCherniavskyTwelbeck},
any element from (\ref{A:partial involution}) is a partial involution.
Following our arguments from~\cite{CanRemmel} for
the idempotents of $\mc{P}_n$, next, we will compute the stabilizer 
subgroup of the partial involution $F_{i,j}$. Let us mention in passing
that for all $i$ in $[n]$, the matrix $E_{i,i}$ is already an idempotent, 
hence, we know its stabilizer subgroup.

\begin{Lemma}\label{L:stabilizer of Fij}
The odun of $F_{i,j}$, $o(F_{i,j})$ is equal to the $\mc{S}_n$-module 
$$
o(F_{i,j})=\oplus_{i,j\in [n],\ i\neq j} \C F_{i,j}.
$$
The stabilizer subgroup of $F_{i,j}$ in $\mc{S}_n$ is isomorphic 
to the parabolic subgroup $\mc{S}_2\times \mc{S}_{n-2}$. 
\end{Lemma}
\begin{proof}
As we already mentioned before, the adjoint (conjugation) action
of $\mc{S}_n$ on partial transformation matrices amounts to 
the permutation action of $\mc{S}_n$ on the labels of the associated 
graph (Lemma~\ref{L:perm on Fij}). Our first claim readily follows from
this argument. 
Now, without loss of generality, we assume that $F_{i,j} = F_{1,2}$. 
In other words, 
\begin{align}\label{A:F12}
F_{i,j} = F_{1,2} = 
\begin{bmatrix}
0 & 1 & 0 \\
1 & 0 & 0 \\
0 & 0 & \mathbf{0}_{n-2}
\end{bmatrix}.
\end{align}
Clearly, the stabilizer subgroup in $\mc{S}_n$ of 
the matrix (\ref{A:F12}) consists of matrices of the form 
$\begin{bmatrix} \sigma_1 & 0 \\ 0 & \sigma_2 \end{bmatrix}$,
where $\sigma_1\in \mc{S}_2$ and $\sigma_2 \in \mc{S}_{n-2}$.
\end{proof}

\begin{proof}[Proof of Theorem~\ref{T:main2}]
It follows from Lemma~\ref{L:stabilizer of Fij} that 
the representation of $\mc{S}_n$ on the orbit $o(F_{i,j}) \cong \mc{S}_n\cdot F_{1,2}$
is isomorphic to the left multiplication action of $\mc{S}_n$ on the 
right coset space $\C [ \mc{S}_n / \mc{S}_2\times \mc{S}_{n-2}]$.
It follows from definitions that this representation is isomorphic to 
\begin{align}\label{A:representation}
\C [ \mc{S}_n / \mc{S}_2\times \mc{S}_{n-2}] \cong \text{Ind}_{\mc{S}_2\times \mc{S}_{n-2}}^{\mc{S}_n} \textbf{1}.
\end{align}
In particular, the Frobenius character of (\ref{A:representation})
is given by $F_{o(F_{i,j})} = s_2 s_{n-2} $. By the Pierri rule, we have 
\begin{align}\label{A:FoF}
F_{o(F_{i,j})} &=
\begin{cases}
s_2 & \text{ if $n=2$;}\\ 
s_3+ s_{2,1} & \text{ if $n=3$;}\\
s_n+ s_{n-1,1} + s_{n-2,2} & \text{ if $n\geq 4$.}  
\end{cases}
\end{align}
The rest of the proof follows from combining (\ref{A:FoF}) 
with the formula (\ref{A:type 2}).

\end{proof}

\begin{proof}[Proof of Corollary~\ref{T:main3}]
The Frobenius character of $\mt{Mat}_n(\C)$ is the sum
of the Frobenius characters of $\mt{Sym}_n(\C)$ and 
$\mt{Skew}_n(\C)$. The proof now is a consequence of Theorems~\ref{T:main1} and~\ref{T:main2}.
\end{proof}

\bibliography{References.bib}

\def\cprime{$'$} \def\cprime{$'$}
\begin{thebibliography}{1}

\bibitem{Can17}
Mahir~Bilen Can.
\newblock A representation on labeled rooted forests.
\newblock {\em Communications in Algebra}, 2018.
\newblock https://doi.org/10.1080/00927872.2018.1439042.

\bibitem{CanCherniavskyTwelbeck}
Mahir~Bilen Can, Yonah Cherniavsky, and Tim Twelbeck.
\newblock Bruhat order on partial fixed point free involutions.
\newblock {\em Electron. J. Combin.}, 21(4):Paper 4.34, 23, 2014.

\bibitem{CanRemmel}
Mahir~Bilen Can and Jeff Remmel.
\newblock Loop-augmented forests and a variant of the foulkes's conjecture.
\newblock https://arxiv.org/abs/1708.02995, 2017.

\bibitem{MR0460128}
Louis Comtet.
\newblock {\em Advanced combinatorics}.
\newblock D. Reidel Publishing Co., Dordrecht, enlarged edition, 1974.
\newblock The art of finite and infinite expansions.

\bibitem{LoehrRemmel}
Nicholas~A. Loehr and Jeffrey~B. Remmel.
\newblock A computational and combinatorial expos\'e of plethystic calculus.
\newblock {\em J. Algebraic Combin.}, 33(2):163--198, 2011.

\bibitem{Mac}
I.~G. Macdonald.
\newblock {\em Symmetric functions and {H}all polynomials}.
\newblock Oxford Mathematical Monographs. The Clarendon Press, Oxford
  University Press, New York, second edition, 1995.
\newblock With contributions by A. Zelevinsky, Oxford Science Publications.

\bibitem{Thrall}
R.~M. Thrall.
\newblock On symmetrized {K}ronecker powers and the structure of the free {L}ie
  ring.
\newblock {\em Amer. J. Math.}, 64:371--388, 1942.

\end{thebibliography}
\bibliographystyle{plain}
\end{document}